\documentclass[11pt]{amsart}
\usepackage{stmaryrd}
\usepackage{fancyhdr, color,mathrsfs,bbm,amscd}
\usepackage[dvips]{graphics}
\usepackage{epsfig,graphicx,euscript}
\usepackage[all]{xy}

\newtheorem{theorem}{Theorem}[section]
\newtheorem{lemma}[theorem]{Lemma}
\newtheorem{prop}[theorem]{Proposition}
\newtheorem{corol}[theorem]{Corollary}
\newtheorem{claim}{Claim}

\theoremstyle{definition}
\newtheorem{defin}[theorem]{Definition}
\newtheorem{exam}[theorem]{Example}

\theoremstyle{remark}
\newtheorem{remk}[theorem]{Remark}

\numberwithin{equation}{section}

 \DeclareMathOperator{\rad}{rad}
  \DeclareMathOperator{\rk}{rk}
 \DeclareMathOperator{\tors}{tors}
 
 \DeclareMathOperator{\VB}{VB}
  
 \DeclareMathOperator{\ev}{ev}
 \DeclareMathOperator{\id}{Id}
 \DeclareMathOperator{\CM}{CM}
 \def\hom{\mathop\mathrm{Hom}\nolimits}
 \DeclareMathOperator{\ext}{Ext}
 
 \DeclareMathOperator{\End}{End}
 \DeclareMathOperator{\cen}{cen}
\DeclareMathOperator{\spec}{Spec}
\DeclareMathOperator{\gd}{gl.dim}
\DeclareMathOperator{\coh}{Coh}
\DeclareMathOperator{\qoh}{Qcoh}
\DeclareMathOperator{\supp}{supp}
\DeclareMathOperator{\aut}{Aut}
\DeclareMathOperator{\br}{Br}

 \DeclareMathOperator{\im}{Im}
 \def\Hom{\mathop{\mathcal{H}\!\mathit{om}}\nolimits}
 \def\Ext{\mathop{\mathcal{E}\!\mathit{xt}}\nolimits}
 \def\Aut{\mathop{\mathcal{A}\!\mathit{ut}}\nolimits}
 \def\proj{\mathop\mathrm{Proj}\nolimits}

\def\emb{\hookrightarrow}
\def\sh{^\sharp}	\def\*{\otimes}
\def\={\setminus}	\def\Arr{\Rightarrow}
\def\+{\oplus}		
\def\sb{\subset}	\def\sbe{\subseteq}
	
\def\ch{^\vee}		\def\dg{^\dagger}
\def\ito{\stackrel\sim\to}
\def\ol{\overline}
\def\xarr{\xrightarrow}
\def\ilim{\varprojlim}
\def\dlim{\varinjlim}

\def\red{_\mathrm{red}}

 \def\rH{\mathrm{H}}	\def\rD{\mathrm D}
\def\fK{\mathbf k}	\def\kH{\mathcal H}
\def\kK{\mathcal K}	
\def\kI{\mathcal I}	\def\kO{\mathcal O}	
	\def\nH{\mathrm h}
\def\vi{\varphi}	\def\Ga{\Gamma}
\def\om{\omega}		\def\La{\Lambda}
\def\al{\alpha}		\def\be{\beta}
\def\eps{\varepsilon}	\def\th{\theta}
\def\de{\delta}		
\def\tX{\tilde X}	\def\tA{\tilde A}
\def\oX{\breve X}	
\def\kF{\mathcal F}	\def\kG{\mathcal G}
\def\kM{\mathcal M}	
	\def\la{\lambda}
\def\tD{\tilde D}	\def\8{\infty}
\def\dX{\mathfrak X}	\def\dA{\mathfrak A}

\def\lb{\textup{(}}	\def\rb{\textup{)}}

\def\gnrsuch#1#2{\langle\,#1\mid #2\,\rangle}
\def\set#1{\left\{\,#1\,\right\}}
\def\gnr#1{\langle\,#1\,\rangle}

\def\ncsm{non-commu\-ta\-tive scheme}
\def\ncv{non-commu\-ta\-tive variety}
\def\ncs{non-commutative surface}
\def\ncss{non-commu\-ta\-tive surface singularity}
\def\cm{Cohen--Macaulay}
\def\vb{vector bundle}
\def\iff{if and only if }
\def\glg{globally generated}
\def\ggg{generically globally generated}
\def\ncc{non-commutative curve}
\def\wrc{weak reduction cycle}
\def\rec{reduction cycle}
\def\oc{one-to-one correspondence}

\title[Cohen-Macaulay modules over non-commutative singularities]{On Cohen--Macaulay modules 
over non-commutative surface singularities}

\author[Y. Drozd]{Yuriy A. Drozd}
\author[V. Gavran]{Volodymyr S. Gavran}
\address{Institute of Mathematics, National Academy of Sciences of Ukraine,
Tereschenkivska str. 3, 01601 Kyiv, Ukraine}
\email{y.a.drozd@gmail.com, drozd@imath.kiev.ua}
\urladdr{www.imath.kiev.ua/$\sim$drozd}
\email{v.gavran@yahoo.com}

\begin{document}
\begin{abstract}
 We generalize the results of Kahn about a correspondence between \cm\ modules and \vb s to
 non-commutative surface singularities. As an application, we give examples of non-commutative 
 surface singularities which are not \cm\ finite, but are \cm\ tame.
\end{abstract}

\maketitle
 
\tableofcontents

\section*{Introduction}
\label{intro} 

 \cm\ modules over commutative \cm\ rings have been widely studied. A good survey on this topic is the 
 book of Yoshino \cite{yo}. In particular, for curve, surface and hypersurface singularities  
 criteria are known for them to be \emph{\cm\ finite}, i.e. only having finitely many
 indecomposable \cm\ modules (up to isomorphism). For curve singularities and minimally elliptic
 surface singularities criteria are also known for them to be \emph{\cm\ tame}, i.e. only having
 1-parameter families of non-isomorphic indecomposable \cm\ modules \cite{dg,dgk}. Less is known
 if we consider non-commutative \cm\ algebras. In \cite{dk} a criterion was given for a \emph{primary}
 1-dimensional \cm\ algebra to be \cm\ finite. In \cite{art} (see also \cite{ch}) a criterion of
 \cm\ finiteness is given for \emph{normal} 2-dimensional \cm\ algebras (maximal orders). 
 As far as we know, there are no examples of 2-dimensional \cm\ algebras which are not \cm\ finite 
 but are \cm\ tame.
 
 In this paper we use the approach of Kahn \cite{Kahn} to study \cm\ modules over normal 
 non-commutative surface singularities. Just as in \cite{Kahn}, we establish (in  Section 2)
 a \oc\ between such modules and \vb s over some, in general non-commutative, projective 
 curves (Theorem~\ref{the}). In Sections 3 and 4 we apply this result to a special case,
 which we call ``\emph{good elliptic}.'' It is analogous to the minimally elliptic case in 
 \cite{Kahn}, though seems somewhat too restrictive. Unfortunately, we could not find more general
 conditions which ensure such analogy. As an application, we present two examples of \cm\ tame
 non-commutative surface singularities (Examples~\ref{ex1}~and~\ref{ex2}). We hope that this approach
 shall be useful in more general situations too.

\section{Preliminaries}
\label{sec1} 

 We fix an algebraically closed field $\fK$, say \emph{algebra} instead of $\fK$-algebra,
 \emph{scheme} instead of $\fK$-scheme and write $\hom$ and $\*$ instead of $\hom_\fK$ and
 $\*_\fK$. We call a scheme $X$ a \emph{variety} if $\fK(x)=\fK$ for every closed point $x\in X$.

\begin{defin}\label{ncv} 
 A \emph{\ncsm} is a pair $(X,A)$, where $X$ is a scheme and $A$ is a sheaf of $\kO_X$-algebras
 coherent as a sheaf of $\kO_X$-modules. If $X$ is a variety, $(X,A)$ is called a \emph{\ncv}.
 We say that $(X,A)$ is \emph{affine, projective, excellent}, etc. if so is $X$.
 
 A \emph{morphism} of \ncsm s $(X,A)\to (Y,B)$ is their morphism as ringed spaces, i.e. a 
 pair $(\vi,\vi\sh)$, where $\vi:X\to Y$ is a morphism of schemes and $\vi\sh:\vi^{-1}A\to B$ 
 is a morphism of sheaves of algebras. A morphism $(\vi,\vi\sh)$ is said to be \emph{finite, 
 projective} or \emph{proper} if so is $\vi$. We often omit $\vi\sh$ and write 
 $\vi:(X,A)\to(Y,B)$.
 
 For a \ncsm\ $(X,A)$ we denote by $\coh A$ ($\qoh A$) the category of coherent (quasi-coherent)
 sheaves of $A$-modules. Every morphism $\vi:(X,A)\to(Y,B)$ induces functors of direct image
 $\vi_*:\qoh A\to \qoh B$ and inverse image $\vi^*:\qoh B\to \qoh A$, where 
 $\vi^*\kF=A\*_{\vi^{-1}B}\vi^{-1}\kF$. Note that this inverse image does not coincide with the
 inverse image of sheaves of $\kO_X$-modules. The latter (when used) will be denoted by $\vi_X^*$.
 Note also that $\vi^*$ maps coherent sheaves to coherent. The pair $(\vi^*,\vi_*)$ is a pair of 
 adjoint functors, i.e. there is a functorial isomorphism 
 $\hom_A(\vi^*\kF,\kG)\simeq\hom_B(\kF,\vi_*\kG)$ for any sheaf of $B$-modules $\kF$ and any sheaf
 of $A$-modules $\kG$.
 
 We call a coherent sheaf of $A$-modules $\kF$ a \emph{\vb} if it is locally projective,
 i.e. $\kF_p$ is a projective $A_p$-module for every point $p\in X$. We denote by $\VB(A)$ the full
 subcategory of $\coh A$ consisting of \vb s.
  
 A \ncsm\ $(X,A)$ is said to be \emph{regular} if $\gd A_p=\dim_pX$ for every point $p\in X$
 (it is enough to check this property at the closed points).
 
 We say that $(X,A)$ is \emph{reduced} if $X$ is reduced and neither stalk $A_p$ contains
 nilpotent ideals. Then, if $\kK=\kK_X$ is the sheaf of rational functions on $X$, 
 $\kK(A)=A\*_{\kO_X}\kK$ is a locally constant sheaf of semisimple $\kK$-algebras. We call it 
 the \emph{sheaf of rational functions}
 on $(X,A)$. In this case each stalk $A_p$ is an \emph{order} in the algebra $\kK(A)_p$, i.e.
 an $\kO_{X,p}$-algebra finitely generated as $\kO_{X,p}$-module and such that $\kK_p A_p=\kK(A)_p$.
 We say that $(X,A)$ is \emph{normal} if $A_p$ is a maximal order in $\kK(A)_p$ for each $p$.
 Note that a regular scheme is always reduced, but not necessarily normal. 

  A morphism $(\vi,\vi\sh):(X,A)\to(Y,B)$ of reduced \ncsm s is said to be \emph{birational}
 if $\vi:X\to Y$ is birational and the induced map $\kK(B)\to\kK(A)$ is an isomorphism.
 
 A \emph{resolution} of a \ncsm\ $(X,A)$ is a proper birational morphism 
 $(\pi,\pi\sh):(\tX,\tA)\to(X,A)$, where $(\tX,\tA)$ is regular and normal.
\end{defin}

\begin{remk}
 Let $(X,A)$ be a \ncsm\ and $C=\cen(A)$ be the center of $A$. (It means that $C_p=\cen(A_p)$
 for every point $p\in X$.) Let also $X'=\spec C$. The natural morphism $\vi:X'\to X$ is finite
 and $A'=\vi^{-1}A$ is a sheaf of $\kO_{X'}$-modules, so we obtain a morphism 
 $(\vi,\vi\sh):(X',A')\to(X,A)$, where $\vi\sh$ is identity. Moreover, the induced functors
 $\vi^*$ and $\vi_*$ define an equivalence of $\qoh A$ and $\qoh A'$. So, while we are interesting
 in study of sheaves, we can always suppose that $A$ is a sheaf of \emph{central $\kO_X$-algebras}.
 Note that if $(X,A)$ is normal and $A$ is central, then $X$ is also normal.
\end{remk}

 Given a \ncsm\ $(X,A)$ and a morphism of schemes $\vi:Y\to X$, we can consider the \ncsm\
 $(Y,\vi^*_YA)$ and uniquely extend $\vi$ to the morphism $(Y,\vi^*_YA)\to(X,A)$ which we 
 also denote by $\vi$. Especially, if $\vi$ is a blow-up of a subscheme of $X$, we call the
 morphism $(Y,\vi^*_YA)\to(X,A)$ the blow-up of $(X,A)$.

\begin{defin}\label{surface} 
 A reduced excellent \ncv\ $(X,A)$ is called a \emph{\ncs} if $X$ is a surface, i.e. $\dim X=2$. 
 If $X=\spec R$, where $R$ is a local complete noetherian algebra with the residue field 
 $\fK$ (then it is automatically excellent), we say that $(X,A)$ is a \emph{germ of \ncss} or, 
 for short, a \emph{\ncss}. In what follows, we identify a \ncss\ $(X,A)$ with the $R$-algebra 
 $\Ga(X,A)$ and the sheaves from $\qoh A$ with modules over this algebra (finitely generated 
 for the sheaves from $\coh A$). 
\end{defin}

 If $(X,A)$ is a \ncs, there always is a normal \ncs\
 $(X',A')$ and a finite birational morphism $\nu:(X',A')\to(X,A)$.
 We call $(X',A')$, as well as the morphism $\nu$, a \emph{normalization} of $(X,A)$. Note that,
 unlike the commutative case, such normalization is usually not unique. 
 
  Let $(X,A)$ be a connected central \ncs\ such that $X$ is normal, 
  $C\sb X$ be an irreducible curve with the general point $g$, $\kK_C(A)=A_g/\rad A_g$ and 
  $\fK_A(C)=\cen\kK_C(A)$. $A$ is normal \iff it is \cm\ (or, the same, reflexive) as a sheaf of 
  $\kO_X$-modules, $\kK_C(A)$ is a simple 
  algebra and $\rad A_g$ is a principal left (or right) $A_g$-ideal for every curve $C$ \cite{rei}. 
  $\fK_A(C)$ is a finite extension of the field of rational functions 
  $\fK(C)=\kO_{X,g}/\rad\kO_{X,g}$ on the curve $C$. 
  The integer $e_C(A)=\dim_{\fK(C)}\fK_A(C)$ is called the \emph{ramification 
 index} of $A$ on $C$, and $A$ is said to be \emph{ramified on $C$} if $e_C(A)>1$. If $p$ is a 
 regular closed point of $C$, we denote by $e_{C,p}(A)$ the ramification index of the extension
 $\fK_A(C)$ over $\fK(C)$ with respect to the discrete valuation defined by the point $p$.
 For instance, if $C$ is smooth, $e_{C,p}(A)$ is defined for all closed points $p\in C$.
 We denote by $D(A)$ the \emph{ramification divisor} $D=D(A)$ which is the union of all 
 curves $C\sb X$ such that $A$ is ramified on $C$. Note that if $p\in X\setminus D(A)$, then 
 $A_p$ is an Azumaya algebra over $\kO_{X,p}$.

 Suppose that $(X,A)$ is a normal \ncs\ and $A$ is central. Then $X$ is \cm\ and $A$ is maximal
 \cm\ as a sheaf of $\kO_X$-modules. We denote by $\CM(A)$ the category of sheaves of maximal \cm\
 $A$-modules, i.e. the full subcategory of $\coh A$ consisting of sheaves $\kF$ which are maximal
 \cm\ considered as sheaves of $\kO_X$-modules. We often omit the attribute ``maximal'' and just
 say shortly ``\cm\ module.'' Obviously, $\VB(A)\sbe\CM(A)$ and these categories coincide \iff $A$
 is regular. For a sheaf $\kF\in\coh A$ we denote by $\kF\ch$ the sheaf $\Hom_A(\kF,A)$. 
 It always belongs to $\CM(A)$. We also set $\kF\dg=\kF^{\vee\vee}$. There is a morphism of functors
 $\,\id\to{\dg}$, which is isomorphism when restricted onto $\CM(A)$.
 If $\vi:(X,A)\to(Y,B)$ is a morphism of central normal \ncs s, we set $\vi\dg\kF=(\vi^*\kF)\dg$.
 
 It is known that every \ncs\ has a regular resolution. 
 More precisely, we can use the following procedure of Chan--Ingalls \cite{ch}.\!%
 \footnote{\,Note that the term "normal" is used in \cite{ch} in more wide sense, but we only need
 it for our notion of normality.}
 The \ncs\ $(X,A)$ is said to be \emph{terminal} 
 \cite[Definition 2.5]{ch} if the following conditions hold:
 \begin{enumerate}
 \item  $X$ is smooth.
 \item  All irreducible components of $D=D(A)$ are smooth.
 \item  $D$ only has normal crossings (i.e. nodes as singular points).
 \item  At a node $p\in D$, for one component $C_1$ of $D$ containing this point, the field
 $\fK_A(C_1)$ is totally ramified over $\fK(C_1)$ of degree $e=e_{C_1}(A)=e_{C_1,p}(A)$, 
 and for the other component $C_2$ also $e_{C_2,p}(A)=e$.
 \end{enumerate}
 It is shown in \cite{ch} that every terminal \ncs\ is regular and every \ncs\ $(X,A)$ has a 
 terminal resolution $\pi:(\tX,\tA)\to(X,A)$.
 Moreover, such resolution can be obtained by a sequence
 of morphisms $\pi_i$, where each $\pi_i$ is either a blow-up of a closed point or a normalization.
 Then $\pi$ is a projective morphism. If $(X,A)$ is a normal \ncss, $\oX=X\=\{o\}$, where $o$ is 
 the unique closed point of $X$, the restriction of $\pi$ onto $\pi^{-1}(\oX)$ is an isomorphism
 and we always identify $\pi^{-1}(\oX)$ with $\oX$. The subscheme $E=\pi^{-1}(o)\red$ is a
 connected (though maybe reducible) projective curve called the \emph{exceptional curve} of the 
 resolution $\pi$. 
 
 Recall also that, for a normal \ncss\ $A$, the category $\CM(A)$, as well as the ramification data
 of $A$, only depends on the algebra $\kK(A)$ \cite[(1.6)]{art}. If $A$ is central and 
 \emph{connected}, i.e. indecomposable as a ring, $\kK(A)$ is a central simple algebra over 
 the field $\kK$, so the category $\CM(A)$ is defined by the class of $\kK(A)$ in the Brauer 
 group $\br(\kK)$, and this class is completely characterized by its ramification data. 
 
 We also use the notion of \emph{non-commutative formal scheme}, which is a pair $(\dX,\dA)$,
 where $\dX$ is a ``usual'' (commutative) formal scheme and $\dA$ is a sheaf of $\kO_\dX$-algebras
 coherent as a sheaf of $\kO_\dX$-modules. If $(X,A)$ is \ncsm\ and $Y\sb X$ is a closed subscheme,
 the \emph{completion} $(\hat X,\hat A)$ of $(X,A)$ along $Y$ is well-defined and general properties
 of complete schemes and their completions, as in \cite{ha,ega}, hold in non-commutative case too.
 
 \section{Kahn's reduction}
 \label{sec2}
 
  From now on we consider a normal \ncss\ $(X,A)$ and suppose $A$ central. We fix a resolution
  $\pi:(\tX,\tA)\to(X,A)$, where $\tA$ is also supposed central. Then $\CM(\tA)=\VB(\tA)$ and we
  consider $\pi\dg$ as a functor $\CM(A)\to\VB(\tA)$. A \vb\ $\kF$ is said to be \emph{full} if
  it is isomorphic to $\pi\dg M$ for some (maximal) \cm\ $A$-module $M$. We denote by $\VB^f(\tA)$
  the full subcategory of $\VB(\tA)$ consisting of full \vb s. We also set 
  $\om_{\tA}=\Hom_{\tX}(\tA,\om_{\tX})$, where $\om_{\tX}$ is a canonical sheaf over $\tX$, and call
  $\om_{\tA}$ the \emph{canonical sheaf} of $\tA$. It is locally free, i.e. belongs to $\VB(\tA)$.
  
  Given a coherent sheaf $\kF\in\coh\tA$, we denote by $\ev_\kF$ the natural map
  $\Ga(\tX,\kF)\*\tA\to\kF$.\!%
\footnote{\,Recall that $\Ga(\tX,\kF)\simeq\hom_{\tA}(\tA,\kF)$.}
   We say that $\kF$ is \emph{\glg} if $\im\ev_\kF=\kF$ and
  \emph{\ggg} if $\supp(\kF/\im\ev_\kF)$ is discrete, i.e. consists of finitely many closed points.
  
 \begin{theorem}[Cf. {\cite[Proposition 1.2]{Kahn}}]\label{full}
 \begin{enumerate}
 \item The functor $\pi\dg$ establishes an equivalence between the categories $\CM(A)$ and
 $\VB^f(\tA)$, its quasi-inverse being the functor $\pi_*$.
 \item A \vb\ $\kF\in\VB(\tA)$ is full \iff the following conditions hold:
 \begin{enumerate}
 \item $\kF$ is \ggg.
 \item The restriction map $\Ga(\tX,\kF)\to\Ga(\oX,\kF)$ is surjective,\\
  \emph{or equivalently, using local cohomologies,}
 \item[(b$'$)]  The map $\al_\kF:\rH^1_E(\tX,\kF)\to\rH^1(\tX,\kF)$ is injective.
 \end{enumerate}
  Under these conditions $\kF\simeq\pi\dg\pi_*\kF$.
 \end{enumerate}
 \end{theorem}
 \begin{proof}
  Note that there is an exact sequence
 \[
  0\to\tors(\pi^*M)\to\pi^*M \xarr{\gamma_M}\pi\dg M \to \ol M\to 0, 
 \]
 where $\tors(\kM)$ denotes the periodic part of $\kM$ and the support of 
 $\ol M$ consists of finitely many closed points. Since $\pi^*M$ is
 always \glg, so is also $\im\gamma_M$. Therefore, $\pi\dg M$ is \ggg. If $M$ is \cm, the
 restriction map $\Ga(X,M)\to\Ga(\oX,M)=\Ga(\oX,\pi^*M)$ is an isomorphism. Since $M$
 naturally embeds into $\Ga(\tX,\pi^*M)$ and hence into $\Ga(\tX,\pi\dg M)$, the 
 restriction $\Ga(\tX,\pi\dg M)\to\Ga(\oX,\pi\dg M)$ is surjective. 
 
 Suppose now that the conditions (a) and (b) hold. Set $M=\pi_*\kF$. Since $\pi$ is projective, 
 $M\in\coh A$. The condition (b) implies that $M\in\CM(A)$. Note that $\Ga(X,M)=\Ga(\tX,\kF)$, 
 so the image of the natural map $\pi^*M\to\kF$ coincides with $\im\gamma_M$. As $\kF$ is 
 \ggg, it implies that the natural map $\pi\dg M\to\kF$ is an isomorphism. It proves (2).
 
 Obviously, the functors $\pi\dg:CM(A)\to\VB^f(\tA)$ and $\pi_*:\VB^f(\tA)\to\CM(A)$ are adjoint. 
 Moreover, if $M=\pi_*\kF$, where $\kF\in\VB^f(\tA)$, there are functorial isomorphisms
 \begin{align*}
  \hom_A(M,M)&\simeq \hom_{\tA}(\pi^*\pi_*\kF,\kF)\simeq\\
   &\simeq\hom_{\tA}(\pi\dg\pi_*\kF,\kF)\simeq \hom_{\tA}(\kF,\kF).
 \end{align*}
 It proves (1).
 \end{proof} 

\begin{remk}
 A full \vb\ over $\tA$ need not be \ggg\ as a sheaf of $\kO_{\tX}$-modules. Moreover, examples
 below show that even the sheaf $\tA=\pi^*A=\pi\dg A$ need not be \ggg\ as a sheaf
 of $\kO_{\tX}$-modules.
\end{remk}

\begin{defin}\label{La} 
  From now on we consider a sheaf of ideals $\kI$
 in $\tA$ such that $\supp(\tA/\kI)\sbe E$, $\La=\tA/\kI$ and
 $Z=\spec(\cen\La)$. Then $(Z,\La)$ is a projective \ncc, i.e. a projective \ncv\ of 
 dimension $1$ (maybe non-reduced). We set $\om_Z=\Ext^1_{\tX}(\kO_Z,\om_{\tX})$ and 
 $$
 \om_\La=\Ext^1_{\tA}(\La,\om_{\tA})\simeq\Ext^1_{\tX}(\La,\om_{\tX})\simeq\Hom_Z(\La,\om_Z).
 $$
\end{defin}

 The sheaves $\om_Z$ and $\om_\La$, respectively, are canonical sheaves for $Z$ and $\La$.
 It means that there are Serre dualities
 \begin{align*}
  \ext^i_Z(\kF,\om_Z)&\simeq\rD\rH^{1-i}(E,\kF)\ \text{ for any } \kF\in\coh Z,\\
  \ext^i_\La(\kF,\om_\La)&\simeq\rD\rH^{1-i}(E,\kF)\ \text{ for any } \kF\in\coh \La,
 \end{align*}
 where $\rD V$ denotes the vector space dual to $V$.
  
 \begin{defin}\label{bi} 
 We say that an ideal $I$ of a ring $R$ is \emph{bi-principal} if $I=aR=Ra$ for a
 non-zero-divisor $a\in R$. A sheaf of ideals $\kI\sb\tA$ is said to be \emph{locally bi-principal} 
 if every point $x\in X$ has a neighbourhood $U$ such that the ideal $\Ga(U,\kI)$ is bi-principal in 
 $\Ga(U,\tA)$. 
 \end{defin}

 \begin{lemma}\label{omega} 
 If the sheaf of ideals $\kI$ is locally bi-principal, then 
 $$
 \om_\La\simeq\Hom_{\tA}(\kI,\om_{\tA})\*_{\tA}\La.
 $$
 \end{lemma}
\begin{proof}
 Let $\kI'=\Hom_{\tA}(\kI,\om_{\tA})$.
 Consider the locally free resolution $0\to\kI\xarr\tau\tA\to\La\to0$ of $\La$. Since $\om_{\tA}$
 is locally free over $\tA$, it gives an exact sequence
 \[
 0\to \om_{\tA}\xarr{\tau^*}\kI'\to\Ext^1_{\tA}(\La,\om_{\tA})\to 0. 
 \]
 On the other hand, tensoring the same resolution with $\kI'$ gives an exact sequence
 \[
 0\to\kI'\*_{\tA}\kI\xarr{1\*\tau}\kI'\to \kI'\*_{\tA}\La\to 0.
 \]
 Since $\kI$ is locally bi-principal, the natural map $\kI'\*_{\tA}\kI\to\om_{\tA}$ is an isomorphism,
 and, if we identify $\kI'\*_{\tA}\kI$ with $\om_{\tA}$, $1\*\tau$ identifies with $\tau^*$.
 It implies the claim of the Lemma.
\end{proof}

\begin{defin}\label{reduct} 
 Let $\kI\sb\tA$ be a bi-principal sheaf of ideals such that $\supp(\tA/\kI)=E$, $\La=\tA/\kI$ and
 $I=\kI/\kI^2$. (Note that $I\in\VB(\La)$.) $\kI$ is said to be a \emph{\wrc} if
 \begin{enumerate}
 \item  $I$ is \ggg\ as a sheaf of $\La$-modules.
 \item  $\rH^1(E,I)=0$.
 \end{enumerate}
 If, moreover,
  \vskip.5ex
 \hskip1ex (3)\hskip1ex $\om_\La\ch=\Hom_\La(\om_\La,\La)$ is \ggg\ over $\La$,
 \vskip.5ex\noindent
 $\kI$ is called a \emph{\rec}. 
 
 For a \wrc\ $\kI$ we define the \emph{Kahn's reduction functor} 
 $R_\kI:\CM(A)\to\VB(\La)$ as
 \[
  R_\kI(M) = \La\*_{\tA}\pi\dg M.
 \]
\end{defin}
 
 We fix a \wrc\ $\kI$ and keep the notation of the preceding Definition.
 We also set $\La_n=\tA/\kI^n$, $I_n=\kI^n/\kI^{n+1}$, $\kI^{-n}=(\kI^n)\ch$
 and $I_{-n}=\kI^{-n}/\kI^{1-n}$. 
 In particular, $\La_1=\La$ and $I_1=I$. One easily sees that 
 $I_n\simeq I\*_\La I\*_\La \dots\*_\La I$ ($n$ times) and 
 $I_{-n}\simeq I_n\ch=\Hom_\La(I_n,\La)$.

\begin{prop}\label{zero} 
 If a coherent sheaf $F$ of $\La$-modules is \ggg, then $\rH^1(E,I\*_\La F)=0$.
 In particular, $\rH^1(E,I_n)=0$
\end{prop}
\begin{proof}
 Let $H=\Ga(E,F)$. Consider the exact sequence 
 \[
 0\to N\to H\*\La\to F \to T \to 0,
 \]
 where $N=\ker\ev_F$ and $\supp T$ is $0$-dimensional. It gives the exact sequence
 \[
 0\to I\*_\La N \to H\*I\to I\*_\La F\to I\*_\La T\to0.  
 \]
 Since $\rH^1(E,H\*I)=\rH^1(E,I\*_\La T)=0$, we get that $\rH^1(E,I\*_\La F)=0$.
\end{proof}

 For any \vb\ $\kF$  over $\tA$ set $F=\La\*_{\tA}\kF$ and $F_n=\La_n\*_{\tA}\kF$.
 There are exact sequences
\begin{align}
 &0\to I_n\to\La_{n+1}\to \La_n\to 0, \notag\\
 &0\to I_n\*_\La F\to F_{n+1}\to F_n\to 0. \label{seq1} 
\end{align}
 For $n=1$, tensoring the second one with $I\ch=\Hom_\La(I,\La)$, we get
 \[
  0\to F\to I\ch\*_{\La_2}F_2\to I\ch\*_\La F\to 0. 
 \]

\begin{prop}\label{FkF} 
 Let $\kI$ be a \wrc\ and $\kF$ be a \vb\ over $\tA$ such that $F$ is \ggg\
 over $\La$. Then $\kF$ is also \ggg\ and $\rH^1(\tX,\kI\*_{\tA}\kF)=0$.
 
 \smallskip
 \emph{Note that if $\kF$ is \ggg\ and $\rH^1(\tX,\kI\*_{\tA}\kF)=0$, then $F$ is also \ggg,
 since the map $\rH^0(\tX,\kF)\to\rH^0(\tX,F)$ is surjective.}
\end{prop}
\begin{proof}
 We first prove the second claim.
 Recall that, by the Theorem on Formal Functions \cite[Theorem III.11.1]{ha},
 \[
\textstyle  \rH^1(\tX,\kI\*_{\tA}\kF)\simeq\ilim_n \rH^1(E,\kI/\kI^n\*_{\tA}\kF).
 \]
 (We need not use completion, since $\rH^1(\tX,\kM)$ is finite dimensional for every 
 $\kM\in\coh\tX$.) Since $\kI/\kI^n$ is filtered by $I_m\ (1\le m<n)$, we have to show
 that $\rH^1(E,I_m\*_{\tA}\kF)=\rH^1(E,I_m\*_\La F)=0$ for all $m$. It follows from
 Proposition~\ref{zero}. 
 
 Note that $\Ga(\tX,\kF)=\Ga(X,\pi_*\kF)$ and $\pi_*\kF$ is \glg, since $X$ is affine.
 Moreover, the sheaves $\kF$ and $\pi_*\kF$ coincide on $\oX$. Hence $\Ga(\tX,\kF)$ generate
 $\kF_p$ for all $p\in\oX$. Therefore, we only have to prove that they generate $\kF_p$
 for almost all points $p\in E$. Since $\supp\La=E$, it is enough to show that the global
 sections of $\kF$ generate $F_p$ for almost all $p\in E$. From the exact sequence
 $0\to \kI\*_{\tA}\kF\to\kF\to F\to 0$ and the equality $\rH^1(\tX,\kI\*_{\tA}\kF)=0$
 we see that the restriction $\Ga(\tX,\kF)\to\Ga(E,F)$ is surjective. Since $F$ is \ggg,
 so is also $\kF$.
\end{proof}

\begin{corol}\label{reduct-2}  
 A locally bi-principal sheaf of ideals $\kI\sb\tA$ is a \wrc\ \iff 
 \begin{enumerate}
 \item $\kI$ is \ggg.
 \item $\rH^1(\tX,\kI)=0$.
 \end{enumerate}
 It is a \rec\ if and only if, moreover, $\om\ch_{\tA}\*_{\tA}\kI$ is \ggg.
\end{corol}
\begin{proof}
 If $\kI$ is a \wrc, (1) and (2) follows from Proposition~\ref{FkF}. Conversely, suppose that
 (1) and (2) hold. Since $\rH^2(\tX,\_)=0$, then $\rH^1(E,I)=0$. Moreover, just as in
 Proposition~\ref{zero}, $\rH^1(\tX,\kI\*_{\tA}\kF)=0$ for any \ggg\ $\kF$. In particular,
 $\rH^1(\tX,\kI^2)=0$. Hence the map $\Ga(\tX,\kI)\to\Ga(E,I)$ is surjective, so $I$ is \ggg.
 
 Now let $\kI$ be a \wrc. Note that, by Lemma~\ref{omega},
 \begin{align*}
  \om\ch_\La&=\Hom_\La(\kI\ch\*_{\tA}\om_{\tA}\*_{\tA}\La,\La)\simeq \\
  			&\simeq\Hom_{\tA}(\kI\ch\*_{\tA}\om_{\tA},\La)
  			\simeq\Hom_{\tA}(\om_{\tA},\kI\*_{\tA}\La)\simeq\\
  			&\simeq\om\ch_{\tA}\*_{\tA}\kI\*_{\tA}\La
  			\simeq\om\ch_{\tA}\*_{\tA}\kI/\om\ch_{\tA}\*_{\tA}\kI^2.
 \end{align*}
 Hence, by Proposition~\ref{FkF}, $\om\ch_\La$ is \ggg\ \iff so is $\om\ch_{\tA}\*_{\tA}\kI$.
\end{proof}

 \begin{prop}\label{exist} 
    A \rec\ always exists.
  \end{prop} 
  \begin{proof}
  Since the intersection form is negative definite on the group of divisors on $\tX$ with support
  $E$ \cite{lip}, there is a divisor $D$ with support $E$ such that $\kO_{\tX}(-D)$ is ample. 
  Therefore, for some $n>0$, $\kI=\tA(-nD)$ as well as $\om\ch_{\tA}(-nD)$ are \ggg and,
  moreover, $\rH^1(\tX,\kI)=0$. Obviously, $\kI$ is bi-principal, so it is a \rec.
  \end{proof}
 
Now we need the following modification of the Wahl's lemma \cite[Lemma B.2]{wahl}.

\begin{lemma}\label{wahl} 
 If $\kF$ is a \vb\ over $\tA$, then
 \[ \textstyle
 \rH^1_E(\tX,\kF)\simeq \dlim_n\rH^0(E,\kI^{-n}\*_{\tA}F_n)
 \]
 Moreover, the natural homomorphisms 
 $$
 \rH^0(E,\kI^{-n}\*_{\tA}F_n)\to\rH^0(E,\kI^{-n-1}\*_{\tA}F_{n+1})
 $$
 are injective.
\end{lemma}
\begin{proof}
 Note that $\rH^1_E(\tX,\kF)\simeq\dlim_n\ext^1_{\tA}(\La_n,\kF)$. Consider the spectral
 sequence $\rH^p(\tX,\Ext^q_{\tA}(\La_n,\kF))\Arr\ext^{p+q}_{\tA}(\La_n,\kF)$. Since
 $\Hom_{\tA}(\La_n,\kF)=0$, the exact sequence of the lowest terms gives an isomorphism
 $\ext^1_{\tA}(\La_n,\kF)\simeq\rH^0(E,\Ext^1_{\tA}(\La_n,\kF))$. Applying 
 $\Hom_{\tA}(\_,\kF)$ to the exact sequence $0\to\kI^n\to\tA\to\La_n\to0$, we get
 the exact sequence
 \[
 0\to \kF=\tA\*_{\tA}\kF\to \Hom_{\tA}(\kI^n,\kF)\simeq \kI^{-n}\*_{\tA}\kF
  \to\Ext^1_{\tA}(\La_n,\kF)\to0,
 \]
 whence $\Ext^1_{\tA}(\La_n,\kF)\simeq (\kI^{-n}/\tA)\*_{\tA}\kF$.
 Moreover, since $\kI^{-n}/\tA\sbe\kI^{-n-1}/\tA$ and $\kF$ is locally projective,
 we get an embedding $(\kI^{-n}/\tA)\*_{\tA}\kF\emb(\kI^{-n-1}/\tA)\*_{\tA}\kF$,
 hence an embedding of cohomologies. It remains to note that 
 $$
 (\kI^{-n}/\tA)\*_{\tA}\kF\simeq (\kI^{-n}/\tA)\*_{\La_n}F_n\simeq\kI^{-n}\*_{\tA} F_n,
 $$ 
 since $\kI^n$ annihilates $\kI^{-n}/\tA$.
 \end{proof}
 
 Since $I\*_\La F\simeq\kI\*_{\tA}F$, there is an exact sequence
 \[
  0\to \kI\*_{\tA}F\to F_2\to F\to 0.
 \]
 Multiplying it with $\kI\ch$, we get an exact sequence
 \begin{equation}\label{seq} 
    0\to F\to \kI\ch\*_{\tA}F_2\to \kI\ch\*_{\tA}F\to 0,
 \end{equation}
 which gives the coboundary map $\th_F:\rH^0(E,\kI\ch\*_{\tA}F)\to\rH^1(E,F)$.
 
\begin{prop}[Cf. {\cite[Proposition 1.6]{Kahn}}]\label{full-2}
 Let $\kI$ be a \wrc. A \vb\ $\kF\in\VB(\tA)$ is full \iff 
 \begin{enumerate}
 \item  $F$ is \ggg\ over $\La$.
 \item  The coboundary map $\th_F$ is injective.
 \end{enumerate}
\end{prop}
\begin{proof}
 Let $\kF$ be \ggg. Since $\rH^1(\tX,\kI)=0$, also $\rH^1(\tX,\kI\*_{\tA}\kF)=0$.
 Therefore, the map $\Ga(\tX,\kF)\to\Ga(E,F)$ is surjective, so $F$ is \ggg. Conversely, if
 $F$ is \ggg, so is $\kF$ by Proposition~\ref{FkF}. Hence, this condition (1) is equivalent to the
 condition (1) of Proposition~\ref{full}. So now we suppose that both $\kF$ and $F$ are \ggg.
 
 Consider the commutative diagram
\[
\begin{CD}
  \rH^1_E(\tX,\kF) @>\al_\kF>> \rH^1(\tX,\kF) \\
   @AiAA	@VpVV \\
  \rH^0(E,\kI\ch\*_{\tA}F) @>\th_F>> \rH^1(E,F)
\end{CD}
\]
 Here $i$ is an embedding from Lemma~\ref{wahl} and $p$ is an isomorphism, since
 $\rH^1(\tX,\kI\*_{\tA}\kF)=0$. If $\kF$ is full, $\al_\kF$ is injective, hence so is $\th_F$.
 
 Conversely, suppose that $\th_F$ is injective. We show that all embeddings
 \begin{equation}\label{embed} 
 \rH^0(E,\kI^{-n}\*_{\tA}F_n) \to\rH^0(E,\kI^{-n-1}\*_{\tA}F_{n+1}) 
 \end{equation}
 from Lemma~\ref{wahl} are actually isomorphisms. It implies that $\al_\kF$ is injective,
 so $\kF$ is full.

 The map \eqref{embed} comes from the exact sequence
 \begin{equation}\label{seq2} 
 0\to \kI^{-n}\*_{\tA}F_n\to\kI^{-n-1}\*_{\tA}F_{n+1} \to \kI^{-n-1}\*_{\tA}F\to0 
 \end{equation}
 obtained from the exact sequence
 \[
  0\to \kI\*_{\tA}F_n\to F_{n+1}\to F\to 0 
 \]
 by tensoring with $\kI^{-n-1}$. So we have to show that the connecting homomorphism
 \[
 \be_n: \rH^0(E,\kI^{-n-1}\*_{\tA}F)\to \rH^1(E,\kI^{-n}\*_{\tA}F_n) 
 \]
 is injective. We actually prove that even the map 
 \[
  \be'_n: \rH^0(E,\kI^{-n-1}\*_{\tA}F)\to \rH^1(E,\kI^{-n}\*_{\tA}F), 
 \] 
 which is the composition of $\be_n$ with the natural map
 $\rH^1(E,\kI^{-n-1}\*_{\tA}F_n)\to\rH^1(E,\kI^{-n}\*_{\tA}F)$ is injective.
 
 Indeed, $\be_0$ coincides with $\th_F$. Since all sheaves $\kI^n$ are \ggg, there is a
 homomorphism $m\tA\to \kI^n$ whose cokernel has a finite support. Taking duals, we get an 
 embedding $\kI^{-n}\emb m\tA$. Tensoring this embedding with the exact sequence \eqref{seq2} 
 for $n=0$ and taking cohomologies, we get a commutative diagram
 \[
\begin{CD}
   \rH^0(E,\kI^{-n-1}\*_{\tA}F) @>\be'_n>> \rH^1(E,\kI^{-n}\*_{\tA}F) \\
     @VVV @VVV \\
   m\rH^0(E,\kI^{-1}\*_{\tA}F) @>>> m\rH^1(E,F)
 \end{CD} 
 \]
 where the second horizontal and the first vertical maps are injective. Therefore, $\be'_n$
 is injective too, which accomplishes the proof.
\end{proof}

 We call a \vb\ $F\in\VB(\La)$ \emph{full} if $F\simeq\La\*_{\tA}\kF$, where $\kF$ is a
 full \vb\ over $\tA$.
 
 \begin{theorem}[Cf. {\cite[Theorem~1.4]{Kahn}}]\label{the} 
  Let $\kI$ be a \wrc. A \vb\ $F\in\VB(\La)$ is full \iff 
  \begin{enumerate}
  \item $F$ is \ggg.
  \item There is a \vb\ $F_2\in\VB(\La_2)$ such that $\La\*_{\tA}F_2\simeq F$ and
  the connecting homomorphism
  $\th_F:\rH^0(E,\kI\ch\*_{\tA}F)\to\rH^1(E,F)$ coming from the exact sequence
  \eqref{seq} is injective.
  \end{enumerate}
  If, moreover, $\kI$ is a \rec, the full \vb\ $\kF\in\VB(\tA)$ such that 
  $\La\*_{\tA}\kF\simeq F$ is unique up to isomorphism. Thus the reduction functor $R_\kI$
  induces a \oc\ between isomorphism classes of \cm\ $A$-modules and
  isomorphism classes of full \vb s over $\La$.
 \end{theorem}
\begin{proof}
 Let $\kI$ be a \wrc, $F$ satisfies (1) and (2). If $U\sb E$ is an affine open subset, 
 there is an exact sequence
 \[
  0\to I_n(U)\to \La_{n+1}(U)\to \La_n(U)\to 0,
 \]
 where the ideal $I_n(U)$ is nilpotent (actually, $I_n(U)^2=0$). Therefore, given a projective
 $\La_n(U)$-module $P_n$, there is a projective $\La_{n+1}(U)$-module $P_{n+1}$ such that
 $\La_n(U)\*_{\La_{n+1}(U)}P_{n+1}\simeq P_n$. Moreover, if $P'_n$ is another projective 
 $\La_n(U)$-module, $P'_{n+1}$ is a projective $\La_{n+1}(U)$-module such that
 $\La_n(U)\*_{\La_{n+1}(U)}P'_{n+1}\simeq P'_n$ and $\vi_n:P_n\to P'_n$ is a homomorphism, it can
 be lifted to a homomorphism $\vi_{n+1}:P_{n+1}\to P'_{n+1}$, and if $\vi_n$ is an isomorphism,
 so is $\vi_{n+1}$ too.
 
 Consider an affine open cover $E=U_1\cup U_2$. Let $P_{2,i}=F_2(U_i)$. Iterating the above procedure,
 we get projective $\La_n(U_i)$-modules $P_{n,i}$ such that 
 $$
 \La_n(U_i)\*_{\La_{n+1}(U_i)}P_{n+1,i}\simeq P_{n,i}
 $$
 for all $n\ge2$. If $U=U_1\cap U_2$, there is an isomorphism $\vi_2:P_{2,1}(U)\ito P_{2,2}(U)$.
 It can be lifted to $\vi_n:P_{n,1}(U)\ito P_{n,2}(U)$ so that the restriction of $\vi_{n+1}$ to
 $P_{n,1}$ coincides with $\vi_n$. Hence there are \vb s $F_n$ over $\La_n$ such that
 $\La_n\*_{\tA}F_{n+1}\simeq F_n$. Taking inverse image, we get a \vb\ $\hat\kF=\ilim_nF_n$ over
 the formal non-commutative scheme $(\hat X,\hat A)$ which is the completion of $(\tX,\tA)$
 along the subscheme $E$. As $\tX$ is projective, hence proper over $X$, 
 $\hat\kF$ uniquely arises as the completion of a \vb\ $\kF$ over $\tA$ such that 
 $\La_n\*_{\tA}\kF\simeq F_n$ for all $n$ (see \cite[Theorem 5.1.4]{ega}). 
 If we choose $F_2$ so that the condition (2) holds, $\kF$ is full by Proposition~\ref{full-2}. 
 Thus $F$ is full as well.
 
 Let now $\kI$ be a \rec, $F$ be a full \vb\ over $\La$ and $F_n$ be \vb s over $\La_n$ such that
 $\La_n\*_{\tA}F_{n+1}\simeq F_n$ for all $n$ and $F_2$ satisfies the condition (2). As we have 
 already mentioned, all choices of $F_n$ are locally isomorphic. Therefore, if we fix one of them, 
 their isomorphism classes are in \oc\ with the cohomology set $\rH^1(E,\Aut F_n)$ \cite{gro}. 
 From the exact sequence \eqref{seq1} we obtain an exact sequence of sheaves of groups
 \[
   0\to \kH \to \Aut F_{n+1} \xarr\rho \Aut F_n\to 0,
 \]
 where $\kH=\ker\rho\simeq \Hom_{\La_n}(F_n,I_n\*_\La F)\simeq \Hom_\La(F,I_n\*_\La F)$. 
 It gives an exact sequence of cohomologies
 \begin{align*}
  0&\to \hom_\La(F_n,I_n\*_\La F)\to \aut F_{n+1}\to \aut F_n\xarr\de \\
   &\to \ext^1_\La(F,I_n\*_\La F)\to \rH^1(E,\Aut F_{n+1})\to \rH^1(E,\Aut F_n).
 \end{align*}
 The isomorphism classes of liftings $F_{n+1}$ of a given $F_n$ are in \oc\
 with the orbits of the group $\aut F_n$ naturally acting on $\ext^1_\La(F,I_n\*_\La F)$.
 \cite[Proposition 5.3.1]{gro}. 
 
 We write automorphisms of $F_n$ in the form $1+\vi$ for $\vi\in\End F_n$. 
 Then $\de(1+\vi)=\de_n(\vi)$, where $\de_n:\hom_{\La_n}(F_n,F_n)\to\ext^1_\La(F,I_n\*_{\La}F)$ 
 is the connecting homomorphism coming from the exact sequence \eqref{seq1}. We restrict $\de_n$ to
 $\hom_\La(F,I_{n-1}\*_\La F)$ (see the same exact sequence, with $n$ replaced by $n-1$). The 
 resulting homomorphism $\de'_n:\hom_\La(F,I_{n-1}\*_\La F)\to \ext^1_\La(F,I_n\*_\La F)$ coincides
 with the connecting homomorphism coming from the exact sequence \eqref{seq} tensored with 
 $\kI^{n-1}$. 
 
 \begin{claim}\label{claim} 
   $\de'_n$ is surjective.
 \end{claim}
 
 Indeed, since $F$, $I_{n-1}$ and $\om\ch_\La$ are \ggg, so is their tensor product. Hence, there is a
 homomorphism $m\La\to\om\ch_\La\*_\La I_{n-1}\*_\La F$, thus also 
 $m\om_\La\to I_{n-1}\*_\La F$ whose cokernel has discrete support. Applying $\hom_\La(F,\_)$, we
 get a commutative diagram
 \[
 \begin{CD}
   m\hom_\La(F,\om_\La) @>>> m\ext^1_\La(F,I\*_\La\om_\La) \\
   @VVV	@V{\eta}VV \\
   \hom_\La(F,I_{n-1}\*_\La F)@>{\de_n'}>> \ext^1_\La(F,I_n\*_\La F),
 \end{CD}
 \]
 where $\eta$ is surjective. Note that the first horizontal map here is the $m$-fold Serre 
 dual $\th^*_F$ to the map
 \[
  \th_F:\hom_\La(I,F)\simeq \rH^0(E,I\ch\*_\La F)\to\ext^1_\La(\La,F)\simeq\rH^1(E,F),
 \]
 which is injective. Therefore, $\th^*_F$ is surjective and so is also $\de'_n$.
 
 If $n>1$, every homomorphism $1+\vi$ with $\vi\in \hom_\La(F,I_{n-1}\*_\La F)$ is invertible.
 Hence $\de$ is surjective and $F_{n+1}$ is unique up to isomorphism. If $n=1$, the set
 $\{\,\vi\in \hom_\La(F,F)\,\mid\,1+\vi \text{ is invertible}\,\}$ is open in 
 $\aut F$. Therefore, its image in $\ext^1_\La(F,I\*_\La F)$ is an open orbit of $\aut F$. 
 If we choose another lifting $F'_2$ of $F$ so that the condition (2) holds, it also gives an open
 orbit. Since there can be at most one open orbit, they coincide, hence $F'_2\simeq F_2$. Now,
 if $\kF$ and $\kF'$ are two full \vb s over $\tA$ such that 
 $\La\*_{\tA}\kF \simeq\La\*_{\tA}\kF'\simeq F$, we can glue isomorphisms
 $\La_n\*_{\tA}\kF \ito\La_n\*_{\tA}\kF'$ into an isomorphism $\kF\ito \kF'$.
\end{proof}

 Claim~\ref{claim} also implies the following result.

\begin{corol}[Cf. {\cite[Corollary 1.10]{Kahn}}]
 If $\kF\in\VB^f(\tA)$, then $\ext^1_{\tA}(\kF,\kF)\simeq\ext^1_\La(F,F)$.
\end{corol}
 We omit the proof since it just copies that from \cite{Kahn}.

\section{Good elliptic case}
\label{sec3} 

 There is one special case when the conditions of Theorem~\ref{the} can be made much simpler.
 It is analogous to the case of \emph{minimally elliptic} surface singularities considered in
 \cite[Section 2]{Kahn}. We are not aware of the full generality when it can be done, so we
 only confine ourselves to a rather restricted situation. Thus the following definition shall be
 considered as very preliminary. It will be used in the examples studied in the next section.
  
 \begin{defin}\label{elliptic} 
 Let $\pi:(\tX,\tA)\to(X,A)$ be a resolution of a \ncss, $\kI$ be a \wrc\ and $\La=\tA/\kI$.
 We say that the \wrc\ $\kI$ is \emph{good elliptic} if $\La\simeq\kO_Z$ where $Z$ is a 
 reduced curve of arithmetic genus $1$ (hence $\om_Z\simeq\kO_Z$). Obviously, then $\kI$
 is a \rec.  If a \ncss\ $(X,A)$ has a resolution
 $(\tX,\tA)$ such that there is a good elliptic reduction cycle $\kI\sb\tA$, we say that $(X,A)$ is
 a \emph{good elliptic \ncss}.
 \end{defin}
 
 \begin{remk}\label{ell-2} 
 One easily sees that being good elliptic is equivalent to fulfilments of the following conditions
 for some resolution:
 \begin{enumerate}
 \item $\nH^1(\tX,\tA)=1$.
 \item There is a \wrc\ $\kI$ such that $\La$ is commutative and reduced.
 \end{enumerate}
 Then $\kI$ is also a reduction cycle.
 \end{remk}
 
 For good elliptic \ncss\ we can state a complete analogue of \cite[Theorem 2.1]{Kahn}. Moreover,
 the proof is just a copy of the Kahn's proof, so we omit it.
 
 \begin{theorem}\label{good} 
 Suppose that $\kI$ is a good elliptic reduction cycle for a resolution $(\tX,\tA)$ of a \ncss\
 $(X,A)$, $\La=\tA/\kI$ and $I=\kI/\kI^2$. A \vb\ $F$ over $\La$ is full \iff $F\simeq G\+m\La$,
 where the following conditions hold:
 \begin{enumerate}
 \item  $G$ is \ggg.
 \item  $\rH^1(E,G)=0$.
 \item  $m\ge \nH^0(E,I\ch\*_\La G)$.\!%
 \footnote{\,If we identify $\La$ with $\kO_Z$, then $I\ch\*_\La G$ is identified with $G(Z)$.}
 \end{enumerate}
 
  If these conditions hold and $M$ is the \cm\ $A$-module such that $F\simeq R_\kI M$, then $M$ is
 indecomposable \iff either $m=\nH^0(E,I\ch\*_\La G)$ or $F=\La$ \lb then $M=A$\rb.

 \end{theorem} 
 
 Now, just as in \cite{dgk} (and with the same proof), we obtain the following result.
 
 \begin{corol}\label{tame-wild} 
 Suppose that $\kI$ is a good elliptic reduction cycle for a resolution $(\tX,\tA)$ of a \ncss\
 $(X,A)$ and $\La=\tA/\kI\simeq\kO_Z$. The \ncss\ $(X,A)$ is \cm\ tame \iff $Z$ is either a smooth 
 elliptic curve or a Kodaira cycle \lb a \emph{cyclic configuration} in the sense of \cite{dgk}\rb. 
 Otherwise it is \cm\ wild.
 \end{corol}
 
 For the definitions of \cm\ tame and wild singularities see \cite[Section 4]{dgk}. Though in this 
 paper only the commutative case is considered, the definitions are completely the same in the 
 non-commutative one. 
  
\section{Examples}
\label{sec4} 

 In what follows we consider \ncs\ singularities $(X,A)$, 
 where $X=\spec R$ and $R=\fK[[u,v]]$. We define
 $A$ by generators and relations. The ramification divisor $D=D(A)$ is then given by one relation
 $F=0$ for some $F\in R$, so it is a \emph{plane curve singularity}. 
 
 When blowing up the closed point $o$, we get the subset 
 $\tX\sbe \proj R[\al,\be]$ given by the equation $u\be=v\al$. We cover it by the affine charts
  $U_1:\,\be\ne0$ and $U_2:\,\al\ne0$, so their coordinate rings are, respectively,
  $R_1=R[\xi]/(u-\xi v)$ and $R_2=R[\eta]/(v-\eta u)$, where $\xi=\al/\be$ and $\eta=\be/\al$.

 \begin{exam}\label{ex1} 
  \[
   A=  R\gnrsuch{x,y}{x^2=v,\,y^2=u(u^2+\la v^2),\,xy+yx=2\eps uv},
  \]
  where $\la\notin\set{0,1}$ and $\eps^2=1+\la$.
  Then $F=uv(u-v)(u-\la v)$, so $D$ is of type $T_{44}$. We set $z=xy$, so $\set{1,x,y,z}$ is an
  $R$-basis of $A$ and $z^2=2\eps uvz-uv(u^2+\la v^2)$. One can check that $\fK_C(A)$ is a   
  field, namely a quadratic extension of $\fK(C)$, for every component of $D$. 
  For instance, if this component is $u=v$, and $g$ is its general point,
  then, modulo the ideal $(u-v)A_g$, $(z-\eps uv)^2=0$, so $z-\eps uv\in\rad A_g$. Moreover,
  \begin{align*}
  (z-\eps uv)^2&=z^2-2\eps uvz+(1+\la)u^2v^2=\\
  			&=-uv(u^2+\la v^2)+(1+\la)u^2v^2=\\
  			&=uv(u-v)(\la v-u).
  \end{align*}
  Since $uv(\la v-u)$ is invertible in $A_g$, $u-v\in(z-\eps uv)A_g$. One easily sees that
  $(z-\eps uv)A_g$ is a two-sided ideal and $A_g/(z-\eps uv)A_g\simeq \fK[[u]][x]/(x^2-u)$ is
  a field. (Note that in this factor $\eps uvx=zx=vy$, so $y=\eps ux$.) Therefore, $(X,A)$ is 
  normal and its ramification index equals $2$ on every component of $D$.
  
  After blowing up the closed point $o\in X$, we get
  \[
  \pi^*A(U_1)\simeq R_1\gnrsuch{x,y}{x^2=v,\,y^2=\xi(\xi^2+\la)v^3,\,xy+yx=2\eps\xi v^2} 
  \]
  and $z^2=2\eps\xi v^2z-\xi v^4(\xi^2+\la)$.
  So we can consider the $R_1$-subalgebra $A_1=\pi^*A(U_1)\gnr{z_1}$ of $\kK(A)$, where
  $z_1=v^{-2}z-\eps\xi$. Note that $y_1=v^{-1}y=xz_1+\eps\xi x\in A_1$.
  \[
  \pi^*A(U_2) \simeq R_2\gnrsuch{x,y}{x^2=\eta u,\,y^2=u^3(1+\la\eta^2),\,xy+yx=2\eps\eta u^2}
  \]
  and $z^2=2\eps\eta u^2z-\eta u^4(1+\eta^2\la)$.
  So we can consider the $R_2$-subalgebra $A_2=\pi^*A(U_2)\gnr{y_2,z_2}$ of $\kK(A)$, where
  $y_2=u^{-1}y,\,z_2=u^{-2}z-\eps\eta$. 
  
  Since $y_2=\eta y_1$ and $z_2=\eta^2 z_1$, $A_1(U_1\cap U_2)=A_2(U_1\cap U_2)$, so we can consider
  the \ncs\ $(\tX,\tA)$, where $\tA(U_1)=A_1,\,\tA(U_2)=A_2$. One can check, just as
  above, that it is normal. Its ramification divisor $\tD$ is given on $U_1$ by the equation
  $\xi v(\xi-1)(\xi-\la)=0$ and on $U_2$ by $u\eta(1-\eta)(1-\la\eta)=0$, so its components are 
  projective 
  lines and have normal crossings. Moreover, $e_C(A)=2$ for every component $C$ of $\tD$, and if
  $x\in C$ is a node of $\tD$, then $e_{C,x}(A)=2$. Hence $(\tX,\tA)$ is a terminal resolution
  of $(X,A)$.
  
  Consider the ideal $\kI\sb\tA$ such that $\kI(U_1)=(x)$ and $\kI(U_2)=(x,y_2)$.
  Note that $\eta y_2=xz_2-\eps\eta x$, $z_2y_2=(1+\la\eta^2)x-\eps\eta y_2$ and 
  $y_2^2=(1+\la\eta^2)u$. Therefore, if $p\in U_2$ and $\eta(p)\ne0$, then 
  $y_2,u\in\tA_px=x\tA_p$, while if $p\in U_2$ and $\eta(p)=0$, then $x,u\in\tA_py_2=y_2\tA_p$. 
  Hence $\kI$ is bi-principal.
 \begin{align*}
   A_1/\kI(U_1)&\simeq \fK[\xi,z_1]/(z_1^2-\xi(\xi-1)(\la-\xi)),\\
   \intertext{and}
   A_2/\kI(U_2)&\simeq \fK[\eta,z_2]/(z_2^2-\eta(1-\eta)(\la\eta-1)),
\end{align*}   
 hence $\La=\tA/\kI\simeq \kO_Z$, where $Z$ is an elliptic curve. Moreover, $x$ is a global section
 of $\kI$, hence of $I=\kI/\kI^2$, and it generates $I_p$ for every point $p\in Z$ except the
 point $\8$ on the chart $U_2$, where $\eta=0$. So $\kI$ is a good elliptic \rec\ and 
 $I\simeq \kO_Z(\8)$. 
 
 Now, by Theorem~\ref{good}, \cm\ modules over $A$ can be obtained as follows. We identify $Z$ with
 $\mathrm{Pic^0}(Z)$ taking $\8$ as the zero point. Denote by
 $G(r,d;p)$ the indecomposable \vb\ over $Z$ of rank $r$, degree $d$ and the Chern class
 $p\in Z=\mathrm{Pic^0}(Z)$ (see \cite{at}). It is \ggg\ \iff either $d>0$ or 
 $d=0,\,r=1$ and $p=\8$. In the latter case
 $G(1,0;\8)\simeq\kO_Z$. Then $I\ch\*_\La G(r,d;p)\simeq G(r,d-r;p)$.
 Moreover,
 \[
 \nH^0(Z,G(r,d;p))=\begin{cases}
   0 & \text{ if } 1\le d<0 \text{ or } d=0 \text{ and } p\ne\8,\\
   1 & \text{ if } d=0 \text{ and } p=\8,\\
   d & \text{ if } d>0.
 \end{cases} 
 \]
  So if $M$ is an indecomposable \cm\ $A$-module and $M\not\simeq A$, then it is uniquely 
  determined by its Kahn reduction $R_\kI M$ which is one of the following \vb s:
 \begin{itemize}
  \item[] $G(r,d;p)$, where $d<r$ or $d=r,\,p\ne\8$; then $\rk M=r$.
  \item[] $G(r,r;\8)\+\kO_Z$, where $r>1$; then $\rk M=r+1$.
  \item[] $G(r,d;p)\+(d-r)\kO_Z$, where $d>r$; then $\rk M=d$.
  \end{itemize} 
 In particular, $A$ is \emph{\cm\ tame} in the sense of \cite{dgk}. Namely, for a fixed rank $r$,
 \cm\ $A$-modules of rank $r$, except one of them, form $2(r-1)$ families parametrized by $Z$ and one
 family parametrized by $Z\=\{\8\}$, arising, respectively, from $G(d,r,p)\ (1\le d<r)$, 
 $G(r',r,p)\ (1\le r'<r)$ and $G(r,r,p)\ (p\ne\8)$. 
 \end{exam}

\begin{exam}\label{ex2} 
  \[
 A=R\gnrsuch{x,y}{x^3=v,\,y^3=u(u-v),\,xy=\zeta yx}, \text{ where } \zeta^3=1,\,\zeta\ne1.
  \]
 Then $F=uv(u-v)$ (the singularity of type $D_4$). Just as above, one can check that $A$ is normal
 and $e_c(A)=3$ for every component $C$ of $D$. After blowing up, on the chart $U_1$ we can consider
 the algebra $A_1=\pi^*A(U_1)\gnr{w_1,z_1}$, where $w_1=v^{-1}y^2,\,z_1=v^{-1}xy$, and on the chart 
 $U_2$ we can consider the algebra $A_2=\pi^*A(U_2)\gnr{w_2,z_2}$, 
 where $w_2=u^{-1}y^2,\,z_2=u^{-1}xy$. 
 Again $A_1(U_1\cap U_2)=A_2(U_1\cap U_2)$, so we can glue them into a \ncs\ $(\tX,\tA)$. 
 One can verify that it is terminal. Let $\kI$ be the locally bi-principal ideal in $\tA$ such that
 $\kI(U_1)=(x)$ and $\kI(U_2)=(x,w_2)$. Then $\tA/\kI\simeq\kO_Z$, where $Z$ is the elliptic curve
 given by the equation $z_1^3=\xi(\xi-1)$ on $U_1$ and by $z_2^3=\eta(1-\eta)$ on $U_2$. Again
 $x$ defines a global section of $\kI$, hence of $I$, and $I\simeq\kO_Z(\8)$, where $\8$ is the point
 on $U_2$ with $\eta=0$. Therefore, $\kI$ is a good elliptic \rec\ and \cm\ modules over $A$
 are described in the same way as in Example~\ref{ex1}.
  In particular, $A$ is also \cm\ tame.
\end{exam}


\begin{thebibliography}{99}
\bibitem{art}
 Artin~M.,
Maximal orders of global dimension and Krull dimension two,
  Invent. Math., 1986, 84, 195--222

 \bibitem{at}
 Atiyah~M., Vector bundles over an elliptic curve,
  Proc. Lond. Math. Soc., 1957, 7, 414--452
 
\bibitem{ch}
 Chan~D, Ingalls~C.,
 The minimal model program for orders over surfaces,
 Invent. Math., 2005, 161, 427--452

 \bibitem{dg}
 Drozd~Y, Greuel~G.-M.,
 Cohen-Macaulay module type, Compos. Math., 1993, 89, 315--338

\bibitem{dgk}
 Drozd~Y, Greuel~G.-M., Kashuba~I.,
 On \cm\ modules on surface singularities,
 Mosc. Math. J., 2003, 3, 397--418

\bibitem{dk}
 Drozd~Y, Kirichenko~V.,
 Primary orders with a finite number of indecomposable representations, 
 Izv. Akad. Nauk SSSR Ser. Mat., 1973, 37, 715--736

\bibitem{ha}
 Hartshorne~R.,
 Algebraic Geometry, Springer-Verlag, New~York--Berlin--Heidelberg, 1977
 
\bibitem{gro}
 Grothendieck~A.,
 A General Theory of Fibre Spaces with Structure Sheaf,
 University of Kansas, Lawrence, 1955
 
 \bibitem{ega}
 Grothendieck~A.,
 \'El\'ements de g\'eom\'etrie alg\'ebrique: III.
 \'Etude cohomologique des faisceaux coh\'erents, Premi\`ere partie.
 Publ. math. I.H.\'E.S., 1961, 11, 5--167
 
\bibitem{Kahn}
 Kahn~C.~P.,
 Reflexive modules on minimally elliptic singularities,
 Math.~Ann., 1989, 285, 141--160
 
\bibitem{lip}
 Lipman~J.,
 Rational singularities, with application to algebraic surfaces and unique factorization,
 Publ. math. I.H.\'E.S., 1969, 36, 195--279

\bibitem{rei}
 Reiner~I., Maximal orders, London Math. Soc. Monogr. Ser., 1975, 5
 
 \bibitem{wahl}
 Wahl~J.,
 Equisingular deformations of normal surface singularities, I,
 Ann.~Math., 1976, 104, 325--356

 \bibitem{yo}
 Yoshino~Y.,
 \cm\ Modules over \cm\ Rings,
 London Math. Soc. Lecture Notes Ser., 1990, 146
 
\end{thebibliography}
\end{document}